\documentclass[10pt]{article}

\usepackage{amsmath,latexsym,amssymb}

\usepackage{bm}

\usepackage{mathrsfs}
\usepackage{color}
\usepackage{amsfonts}
\usepackage{amsthm}
\usepackage[mathcal]{euscript}

\theoremstyle{definition}
\newtheorem{definition}{Definition}[section]

\theoremstyle{remark}

\theoremstyle{plain}
\newtheorem{theorem}[definition]{Theorem}
\newtheorem{lemma}[definition]{Lemma}
\newtheorem{proposition}[definition]{Proposition}
\newtheorem{corollary}[definition]{Corollary}

\numberwithin{equation}{section}

\usepackage{authblk}
\title{Orthogonal polynomials in weighted Bergman spaces}

\author[1]{Erwin Mi\~{n}a-D\'{\i}az \thanks{Corresponding author; Email: minadiaz@olemiss.edu}}
\affil[1]{The University of Mississippi,
Department of Mathematics,
Hume Hall 305,P.~O.~Box 1848,
University, MS 38677-1848, USA.}

\date{\today}                     
\newcommand{\caliP}{\mathcal{P}}

\newcommand{\caliH}{\mathcal{H}}

\newcommand{\cj}{\overline}
\newcommand{\D}{\mathbb{D}}
\newcommand{\T}{\mathbb{T}}

\newcommand{\C}{\mathbb{C}}
\newcommand{\Z}{\mathbb{Z}}
\newcommand{\R}{\mathbb{R}}

\begin{document}
\maketitle

\begin{abstract}
Let $w$ be a weight on the unit disk $\D$ having the form  $$w(z)=|v(z)|^2\prod_{k=1}^s\left|\frac{z-a_k}{1-z\cj{a}_k}\right|^{m_k}\,,\quad m_k>-2,\ |a_k|<1,$$ where $v$ is analytic and free of zeros in $\cj{\D}$, and let $(p_n)_{n=0}^\infty$ be the sequence of polynomials ($p_n$ of degree $n$) orthonormal over $\D$ with respect to $w$.  We give an integral representation for $p_n$ from which it is in principle possible to derive its asymptotic behavior as $n\to\infty$ at every point $z$ of the complex plane, the asymptotic analysis of the integral being primarily dependent on the nature of the first singularities encountered by the function $v(z)^{-1}\prod_{k=1}^s(1-z\cj{a}_k)^{-1}$.  
\end{abstract}

\section{Introduction}
Let $\D=\{z\in\C:|z|<1\}$ be the unit disk, and let $\T:=\partial \D$ be the unit circle. We use the letter $\sigma$ to denote the area measure divided by $\pi$, so that $\sigma(\D)=1$. 

A function $w(z)\geq 0$ defined on $\D$ is said to be a weight, provided that $0<\int_{\D}wd\sigma<\infty$, and that for every disk $D(z_0,\epsilon)$ centered at a point $z_0\in \T$ of radius $\epsilon>0$,
\[
\int_{D(z_0,\epsilon)\cap \D} w(z)d\sigma(z)>0.
\]

We write $\caliH(\D)$ for the linear space of holomorphic functions in $\D$. For a weight $w$ in $\D$, the weighted Bergman space  $A^2_{w}$ is the subspace of $\caliH(\D)$ defined as  
\[
 A^2_{w}:=\left\{f\in \caliH(\D):\int_{\D}|f(z)|^2w(z)d\sigma(z)<\infty\right\}.
 \]

The space $A^2_{w}$ is naturally endowed with the inner product 
\begin{align}\label{innerproductw}
	\langle f,g\rangle_w:=\int_{\D}f(z)\cj{g(z)}w(z)d\sigma(z),
\end{align}
which, by means of the Gram-Schmidt orthogonalization process applied to the powers $(z^n)_{n=0}^\infty$,  produces a unique sequence of polynomials $(p_n)_{n=0}^\infty$, $p_n$ of degree $n$ and positive leading coefficient, such that 
\begin{align}\label{orthoconditions}
	\langle p_n,p_m\rangle_w=\delta_{n,m}, \quad n,m\geq 0.
\end{align}
These polynomials  are often referred to as Bergman polynomials. We will denote the leading coefficient of $p_n$ by $\gamma_n$:
\[
p_n(z)=\gamma_nz^n+\cdots,\quad \gamma_n>0.
\]

 In this paper, we study the asymptotic behavior of $p_n$ as $n\to\infty$ for weights of the form 
\begin{align}\label{defweightw}
	w(z)=|v(z)|^2\prod_{k=1}^s\left|\frac{z-a_k}{1-z\cj{a}_k}\right|^{m_k},\quad s\geq 0,
\end{align}
where  $v$ is analytic and never zero in $\cj{\D}$, and 
\[
m_k\in (-2,\infty)\setminus\{0\}, \quad |a_k|<1,\quad 1\leq k\leq s. 
\]  
The condition that $m_k>-2$ guarantees the integrability of $w$ over $\D$. When $s=0$, the product in \eqref{defweightw}, as well as any other product of the form $\prod_{k=1}^0$, is to be interpreted as being constant $1$. We assume for convenience that $v(0)>0$.

The restrictions $u=w|_\T$, with $w$ as in \eqref{defweightw}, generate the class $\caliP_\T$ of positive analytic weights  on the unit circle, that is, weights $u$ of the form $u(z)=|v(z)|^2$, $z\in\T$, with $v$ analytic and free of zeros in $\cj{\D}$. 
Our investigation is motivated by the type of results obtained in  \cite{andrei} for the sequence $(\varphi_n)_{n=0}^\infty$ of Szeg\H{o} polynomials  corresponding to a weight $u\in \caliP_\T$. These are polynomials satisfying the orthonormality conditions
\[
\int_{\T}\varphi_n(z)\cj{\varphi_m(z)}u(z)|dz|=\delta_{n,m}\,,\quad n,m\geq 0.
\]

It has been shown in \cite{andrei} that for a weight in $\caliP_\T$, the Szeg\H{o} polynomial $\varphi_n$ admits an asymptotic expansion whose terms can be generated by the  successive application of a couple of Cauchy integral operators. The dominant term of the expansion is an integral whose behavior depends on the type of singularities encountered by the analytic continuation of the function $v^*(z)=1/\cj{v(1/\cj{z})}$. When these singularities are isolated or branch points, techniques involving contour deformation and residue computations allow for a precise asymptotic description of $\varphi_n$ inside the unit circle.

To a degree,  the results that we present in this paper can be viewed as an extension to the Bergman setting of those obtained in \cite{andrei}. Our main result (Theorem \ref{maintheorem}) asserts that for a weight as in \eqref{defweightw}, the Bergman polynomial $p_n$ can be asymptotically represented by a Cauchy-type integral whose behavior is influenced by the  singularities of the function $v^*(z)/\prod_{k=1}^s(z-a_k)$. The elegance and applicability of this integral resides in  its ability to encode the behavior of $p_n$ at every point of the complex plane.   

 Previously in \cite{MD2} (see also \cite{MD3}), we established similar results for the subfamily of \eqref{defweightw} consisting of weights $w=|h|^2$,  $h$ a polynomial without zeros in $\T$.  The  proof that we employed at the time was heavily based on constructing (for an $h$ with all its zeros in $\D$) an asymptotic expansion for $p_n$ similar to the one given in \cite{andrei} for Szeg\H{o} polynomials. For more general weights, however, our attempt at constructing these expansions has so far been unsuccessful, thus requiring us to look for a different  approach. Our success this time  comes from  a better understanding  of the reproducing kernel associated with the weights \eqref{defweightw}, and the key idea of iterating a certain identity (i.e., \eqref{polyfirstexpansion}), the details of which are explained later in Section \ref{sectionproofmaintheorem}.

We point out that there are some results in the literature that apply (fully or partially) to the weights under our consideration. In \cite{SimBlob}, Simanek  studied the asymptotic properties of $L^p(\mu)$-extremal polynomials for a broad class of measures $\mu$ supported on analytic regions. For a weight as in \eqref{defweightw}, Theorem 1.2 and Theorem 3.1 of \cite{SimBlob} yield 
 \begin{equation}\label{eq22}
	\lim_{n\to\infty}\frac{p_n(z)}{\sqrt{n} z^n}= \frac{1}{\cj{v(1/\cj{z})}}, \quad |z|> 1,
\end{equation}
the convergence being uniform for $|z|\geq r>1$.

A theorem by Suetin \cite[Thm. 3.3]{Suetin} for continuous, positive weights that satisfy a Lipschitz condition of order $\alpha<1$, establishes \eqref{eq22} with an estimate of $O((\ln n/n)^{\alpha/2})$ for the speed of convergence, which at $z=\infty$ \cite[Thm. 3.1]{Suetin} improves to $O((\ln n/n)^{\alpha})$. A better, sharp estimate will be given in Corollary \ref{firstcorollary} below.  

We also mention an earlier result that (when orthogonality is considered over the unit disk) applies to weights whose restriction to an annulus $\rho<|z|<1$ is of the form $|z^m/g(z)|^2$, with $g$ an analytic function without zeros in $\rho<|z|\leq \infty$, and $m$ an integer. Assuming that $g(\infty)>0$, we have 
\begin{equation*}
	\lim_{n\to\infty}\frac{p_n(z)}{\sqrt{n} z^n}= g(z), \quad |z|> \rho.
\end{equation*} 
 This formula was originally proven by Korovkin \cite{Korovkin} for $m=0$, and was later  extended to any integer value of $m$ in \cite[Sec. 3.3.4]{Smirnov}. We will see that for a weight as in \eqref{defweightw}, \eqref{eq22} also holds in the exterior of some circle of radius $<1$.

\section{Statement of results}
In this section, $w$ always stands for a weight of the form \eqref{defweightw}, and $p_n$ is the corresponding $n$th degree Bergman polynomial. 

Let  
\begin{align}\label{defqandq*}
	q(z):=\prod_{k=1}^s (z-a_k), \quad q^*(z):=\prod_{k=1}^s (1-\cj{a}_kz),
\end{align}
and define $\rho_w\geq 0$ to be the smallest number such that the function 
\[
\frac{1}{v(z) q^*(z)}
\]
can be analytically continued to the disk $|z|<1/\rho_w$. By our assumptions on the weight $w$, it is the case that $\rho_w<1$. 

Equivalently, if we define 
\begin{align}\label{exteriorszegofunction}
	v^*(z):=\frac{1}{\cj{v(1/\cj{z})}},
\end{align}
then $\rho_w$ is the smallest number such that the function $v^*/q$ has an analytic continuation to $|z|> \rho_w$.

The circle $|z|=\rho_w$ defines the boundary where $p_n$ has a marked change in behavior. This can be seen already in a weak sense from the fact that for all but finitely many $z\in \D$,
\[
\limsup_{n\to\infty}|p_n(z)|^{1/n}= \max\{|z|, \rho_w\}.
\]

As we explain below in Section \ref{reproducingkernelsection}, this is a consequence of a connection between the orthonormal polynomials and the reproducing kernel $K_h(z,\zeta)$ of the Bergman space $A^2_h$ associated with the weight 
\[
h(z):=\prod_{k=1}^s\left|\frac{z-a_k}{1-\cj{a}_kz}\right|^{m_k}\,.
\]

This kernel is a function defined on  $\D\times\D$,  analytic in $z$ and $\cj{\zeta}$, uniquely determined by the property that for all $f\in A^2_h$,
\begin{align*}
f(z)=\int_\D f(\zeta)K_h(z,\zeta)h(\zeta)d\sigma(\zeta), \quad z\in \D.	
\end{align*}

A structural formula for $K_h$ has been given in \cite{MD1}, and we shall elaborate on it later in Section \ref{reproducingkernelsection}, but for our immediate purpose, it suffices to note that  
\begin{align*}
L(z,\zeta):=\zeta^{-2}K_{h}(z,1/\cj{\zeta})
\end{align*}
is a rational  function that can be written as  
\begin{align}\label{firstformulaforL}
L(z,\zeta)=\frac{1}{(\zeta-z)^2} +\frac{V(z,\zeta) }{(\zeta-z)q^*(z)q(\zeta)},
\end{align}
with $V(z,\zeta)$ a polynomial in the variables $z$ and $\zeta$. We can now state our main result. 
\begin{theorem}\label{maintheorem} There exists a sequence $(H_n)_{n=0}^\infty$ of analytic functions  in  $\D$, all of them vanishing at the origin, such that:
\begin{enumerate}
\item[i)] $H_n(z)=O(1/n)$ uniformly as $n\to\infty$ on compact subsets of $\D$;
\item[ii)] For every $r$ and $z$ such that $\rho_w<r<1$ and $|z|<r$, 
\begin{align}\label{polymainformula}
v(0)\gamma_np_n(z)={} &\frac{1}{2\pi i v(z)}\int_{|\zeta|=r}(vv^*)(\zeta)L(z,\zeta)\zeta^{n+1}(1+H_n(\zeta)
)d\zeta\,.
\end{align}
\end{enumerate}
\end{theorem}

From \eqref{firstformulaforL}, we see that the integral in \eqref{polymainformula} is  well-defined for the stated values of $z$. Its behavior as $n\to\infty$ changes according to the location of $z$ relative to the critical circle $|z|=\rho_w$. For $z$ outside this circle, \eqref{polymainformula} yields the following strong asymptotic formula.

\begin{corollary}\label{firstcorollary} For every $r>\rho_w$, 
	\begin{align}\label{strongasymptoticformula}
\frac{p_n(z)}{\sqrt{n}\,z^n}=v^*(z)+O(1/n)
\end{align}
uniformly on $|z|\geq r$ as $n\to\infty$.	
\end{corollary}

If we evaluate \eqref{strongasymptoticformula} at $z=\infty$, then we get 
\begin{align}\label{gammanformula}
	\gamma_n=\frac{\sqrt{n}}{v(0)}\left(1 +O\left(1/n\right)\right), \quad (n\to\infty).
\end{align}
But in reality,  we will need to have \eqref{gammanformula} already established before we can prove  \eqref{polymainformula}.

The $O(1/n)$ estimate for the error term in \eqref{strongasymptoticformula} and \eqref{gammanformula} is, generally speaking, best possible, as this estimate becomes exact for the weight $w(z)=|z-b|^{2}$, $|b|>1$. This can be seen from the explicit formulas for $p_n$ and $\gamma_n$ discussed in \cite[Remark 2]{MD2}, \cite[Remark 6]{MD3}. 

In the next section, we touch on a few examples that illustrate how the behavior of $p_n$ in the closed disk $|z|\leq \rho_w$ is affected by the nature of the singularities that the function $v^*/q$ has on the critical circle $|z|=\rho_w$. If these singularities are poles or branch points, then the asymptotic analysis of the integral in \eqref{polymainformula} is relatively straightforward, and has been carried out in detail for integrals of a similar type, for instance, in \cite{andrei} and \cite{MD4} (\cite{andrei} also provides an example with an essential singularity). 
 
The function $L$ that occurs in \eqref{polymainformula} is defined through the kernel $K_h$. From the formula for this kernel given in \cite[Thm. 1.1]{MD1}, we can see that 
\begin{align}\label{secondformulaforL}
	L(z,\zeta)={} &\frac{1}{(\zeta-z)^2} +\sum_{k=1}^s\frac{\frac{m_k}{2}(1-|a_k|^2) }{(\zeta-z)(\zeta-a_k)(1-z\cj{a}_k)}+\frac{J(z,\zeta)}{q^*(z)q(\zeta)}, 
\end{align}
where $J(z,\zeta)$ is a polynomial in the variables $z$ and $\zeta$ of degree at most $2(s-2)$ ($s-2$ in each independent variable). For $0\leq s\leq 2$, $J(z,\zeta)$ is constant:  
\[
J(z,\zeta)=\begin{cases}0,& \ s=0,\ 1,\\
	K_h(0,0)-1-\frac{m_1}{2}(1-|a_1|^2)-\frac{m_2}{2}(1-|a_2|^2), &\ s=2.
\end{cases}
\]

For $s>2$, however, $J(z,\zeta)$ has a somewhat complex structure, but it can be explicitly  written in terms of the values of $K_h(z,\zeta)$ and its derivatives at $(0,0)$. These values can in turn be numerically computed by integration.

\section{Examples} 
\subsection{Case $s=0$}\label{cases=0}

For $s=0$, that is, when $w=|v|^2$ with $v$ analytic and never zero in $\cj{\D}$, \eqref{polymainformula} becomes  
\begin{align}\label{polymainformulanozeros}
	v(0)\gamma_np_n(z)={} &\frac{1}{2\pi i v(z)}\int_{|\zeta|=r}\frac{v(\zeta)v^*(\zeta)\zeta^{n+1}}{(\zeta-z)^2}(1+H_n(\zeta)
	)d\zeta\,.
\end{align}

By the definition of $\rho_w$, the function $v^*$ is analytic in $|z|>\rho_w$,  but it has singularities on the circle $|z|=\rho_w$. For $z$ in the disk $|z|\leq \rho_w$, the behavior as $n\to\infty$ of the integral in \eqref{polymainformulanozeros} 
depends on the nature of these first singularities. For instance, if $v(z)=\prod_{k=1}^{\ell}(1-\cj{b}_kz)^{r_k}$ is a polynomial with all its zeros outside the unit circle (i.e., $0<|b_k|<1$), then $v^*$ is a rational function whose poles are the points $b_k$:
\[
v^*(z)=z^r\prod_{k=1}^{\ell}(z-b_k)^{-r_k} ,\quad r=\sum_{k=1}^\ell r_k.
\]
The integral in \eqref{polymainformulanozeros} can then be computed without difficulty by using the residue theorem,  and we recover the results of \cite{MD3}.

To illustrate a possible approach in the presence of branch points, let us sketch the analysis for the weight   $v(z)=(1-\cj{b}z)^{r}$, $|b|<1$, $r\in \R\setminus\Z$. For this weight, \eqref{polymainformulanozeros} takes the form
\begin{align}\label{algebraicsingular}
	\gamma_np_n(z)\sim{}  &\frac{1}{2\pi i v(z)}\int_{|\zeta|=r}\left(\frac{\zeta-b}{\zeta}\right)^{-r}\frac{v(\zeta)\zeta^{n+1}}{(\zeta-z)^2}d\zeta\,.
\end{align}

 For a number $\mu$ such that $0<\mu<|b|$, let $C_\mu$ be the positively oriented contour that results from attaching the two-sided segment $[\mu b/|b|,b]$ to the circle $|\zeta|=\mu$. For $|z|<\mu$, the integral in \eqref{algebraicsingular} can be taken over $C_\mu$ without affecting its value, and  can further be split as the sum of an  integral over $|\zeta|=\mu$, which is $O(\mu^n)$, plus an integral over the two-sided segment  $[\mu b/|b|,b]$, which provides the dominant behavior as $n\to\infty$. In this way, we get
\begin{align}\label{algebraicsingular1}
	\gamma_np_n(z)\sim {}&\frac{\sin(\pi r)}{\pi  v(z)}\int_{[\mu b/|b|,b]}\left|\frac{\zeta-b}{\zeta}\right|^{-r}\frac{v(\zeta)\zeta^{n+1}}{(\zeta-z)^2}d\zeta.
\end{align}

Noting that  
\begin{align*}
\frac{v(\zeta)}{(\zeta-z)^2}=\frac{v(b)}{(b-z)^2}+O(\zeta-b)
\end{align*}
uniformly as $\zeta\to b$ on $|z|\leq R<\mu$,  and bearing in mind a few basic properties of the Gamma and Beta functions, we can conclude from  \eqref{algebraicsingular1} that 
\begin{align}\label{finaloutcome}
	\begin{split}
	\gamma_np_n(z)\sim{} &\frac{\sin(\pi r)(1-|b|^2)^{r}b^{n+2}}{\pi  (1-\cj{b}z)^{r}(b-z)^2}\int_{\mu/|b|}^1t^{n+1+r}(1-t)^{-r}dt\\
	\sim{} &\frac{(1-|b|^2)^{r}b^{n+2}}{(1-\cj{b}z)^{r}(b-z)^2}\cdot\frac{r(r+1)\cdots(r+n+1) }{(n+2)!} \,.
	\end{split}
\end{align}

These computations actually work if $r<1$, which guarantees integrability over the segment $[\mu b/|b|,b]$. When $r>1$, we first need to integrate by parts in \eqref{algebraicsingular} as many times as needed to bring $r$ into the appropriate range, and then proceed as we just explained (the final outcome is the same as \eqref{finaloutcome}). More detailed explanations can be found, for instance, in the proof of Proposition 1 of  \cite{MD4}.

If $v$ is meromorphic in $\C$ without zeros, then $\rho_w=0$, $v^*$ is an analytic function in $\cj{\C}\setminus \{0\}$, and \eqref{strongasymptoticformula} holds for $|z|>0$. For $v(z)=e^z$, there is a curious relationship between the sequences of Szeg\H{o} and Bergman polynomials.  

\begin{proposition} Let $(\psi_n)_{n=0}^\infty$ be the sequence of monic polynomials that are orthogonal on the unit circle with respect to the weight $|e^z|^2$, and let $p_n$ be the Bergman polynomial of degree $n$ for the same weight.  Then, $\psi_n(0)\not=0$ and
	\begin{align}\label{relationshippolys}
	\gamma_n^{-1}p_n(z)=\frac{\psi^*_n(z)+{\psi^*_n}'(z)}{\cj{\psi_n(0)}},\quad n\geq 0,
	\end{align}
where  $\psi^*_n(z):=z^n\cj{\psi_n(1/\cj{z})}$.
\end{proposition}
\begin{proof} 
Let $\phi_n$ denote the numerator of the fraction that occurs  in \eqref{relationshippolys}.  With the aid of Green's Formula \cite[p. 10]{Gaier}, we find that 
	\begin{align}\label{orthocondforphi}
		\begin{split}
		\int_{\D}\cj{\phi_n(z)} z^m|e^z|^2d\sigma(z)={} &\int_{\D }\cj{\left(e^z\psi^*_n(z)\right)'} z^me^zd\sigma(z)\\
		= {} &\frac{1}{2\pi i}\int_{|z|=1}\cj{e^z\psi^*_n(z)}z^me^zdz\\
		={}& \frac{1}{2\pi}\int_{|z|=1}\cj{\psi^*_n(z)}z^{m+1}|e^z|^2|dz|\,.
\end{split}	
\end{align}

Since the monic polynomial $\psi_n$ is characterized by the orthogonality property that the last integral in \eqref{orthocondforphi} equals zero for all $0\leq m\leq n-1$, it follows that for these values of $m$, $\phi_n$ is orthogonal to $z^m$ over $\D$ with respect to the weight $|e^z|^2$. This implies that $\phi_n$  is of degree $n$, unless $\phi_n$ is identically zero. The latter is impossible because $\psi^*_n$ is not the constant zero. Since $\phi_n$ is then of degree $n$, so is $\psi_n^*$, so that $\psi_n(0)\not =0$, and thus the right-hand side of \eqref{relationshippolys} (which equals $\phi_n/\cj{\psi_n(0)}$) is precisely the monic Bergman orthogonal polynomial $\gamma_n^{-1}p_n$.
\end{proof}

By virtue of \eqref{relationshippolys} and the results of $\cite{andrei}$, $w(z)=|e^z|^2$ is another example of a weight on $\D$ for which a full asymptotic expansion for $p_n$ can be derived.

\subsection{Case $s\geq 1$}
We will assume that 
\[
\rho_a:=\max_{1\leq k\leq s}|a_k|>0,
\]
for otherwise $s=1$ and $a_1=0$, in which case the asymptotic analysis of the integral in \eqref{polymainformula} is practically the same as for $s=0$. 

The weight  
\begin{align}\label{bsweights}
w(z)=\frac{\prod_{k=1}^s|z-a_k|^{m_k}}{\prod_{k=1}^s|1-z\cj{a}_k|^{m_k+r_k}}
\end{align}
with $r_k\in \R$ and $m_k\in (-2,\infty)\setminus\{0\}$, serves as a good prototype for illustrating a few situations that might arise. This weight can be written in the form \eqref{defweightw} with 
\[
v(z)=\prod_{k=1}^s(1-z\cj{a}_k)^{-r_k/2}, \quad v^*(z)=\prod_{k=1}^s\left(\frac{z-a_k}{z}\right)^{r_k/2},
\]
the branches of the powers being chosen so that $v(0)=1$.

If every $r_k$ is a positive, even integer, then $v^*/q$ is analytic in $|z|>0$, so that $\rho_w=0$. Set \[
 d:=\frac{1}{2}\sum_{k=1}^sr_k,\quad p(z):=\prod_{k=1}^s\left(z-a_k\right)^{r_k/2}\,.
\]

From  \eqref{strongasymptoticformula} and Hurwitz's Theorem, we see that as $n\to\infty$, $n-d$ zeros of $p_n$ approach the origin, while each $a_k$ attracts $r_k/2$ zeros, multiplicities being counted. But we can be more precise, for in this case the integral in  \eqref{polymainformula} can be computed directly by using the residue theorem. For  $n+1\geq d$, the computation simplifies into   
\begin{align}\label{bspolys}
	\begin{split}
	\gamma_np_n(z)={} &(n+1-d)z^{n-d}p(z)(1+H_n(z))+z^{n+1-d}p(z)H'_n(z)\\
	&+z^{n+1-d}(1+H_n(z))\frac{p(z)}{2}\sum_{k=1}^s\frac{(m_k+r_k)(1-|a_k|^2) }{(z-a_k)(1-z\cj{a}_k)}\,.
	\end{split}
\end{align}

 It follows from \eqref{bspolys} that for $n\geq d$,
\[
\gamma_np_n(z)=\gamma_n^2z^{n-d}\left(\prod_{k=1}^s(z-a_k)^{\frac{r_k}{2}-1}\right)\prod_{k=1}^s(z-z_{n,k})\,,
\] 
with $z_{n,k}\to a_k$ as $n\to\infty$. 

By substituting this expression for $\gamma_np_n$ back into \eqref{bspolys} and simplifying, we get (with $q$ as in \eqref{defqandq*} and $q_k=q/(z-a_k)$) 
\begin{align}\label{identity}
	\begin{split}
		\gamma_n^2\prod_{k=1}^s(z-z_{n,k})={} &(n+1-d)q(z)(1+H_n(z))+zq(z)H'_n(z)\\
		&+\frac{z(1+H_n(z))}{2}\sum_{k=1}^s\frac{q_k(z)(m_k+r_k)(1-|a_k|^2)}{1-z\cj{a}_k}\,.
	\end{split}
\end{align}

Evaluating at $z=a_j$, we obtain 
\begin{align*}
	\gamma_n^2(a_j-z_{n,j})\underset{k\not=j}{\prod_{1\leq k\leq s}}(a_j-z_{n,k})={} &\frac{a_j(m_j+r_j)}{2}(1+H_n(a_j))\underset{k\not=j}{\prod_{1\leq k\leq s}}(a_j-a_k).
\end{align*}
This and \eqref{gammanformula} shows that $|a_j-z_{n,j}|=O(1/n)$ as $n\to\infty$, which leads  us to conclude that for all $1\leq j\leq s$,
\begin{align*}
z_{n,j}= a_j-\frac{a_j(m_j+r_j)}{2n}+O(n^{-2})\quad (n\to\infty). 
\end{align*}

Since $H_n(0)=0$, evaluating \eqref{identity} at $0$ results in
\begin{align*}
	\begin{split}
		\gamma_n^2={} &(n+1-d)\prod_{k=1}^s\frac{a_k}{z_{n,k}}\\
		={} &(n+1-d)\left(1+\frac{1}{2n}\sum_{k=1}^s(m_k+r_k)+O(n^{-2})\right)\,.
	\end{split}
\end{align*}

Notice the resemblance of $p_n$ to the so-called Bernstein-Szeg\H{o} polynomial $z^{n-d}\prod_{k=1}^s(z-a_k)^{\frac{r_k}{2}}$, $n\geq d$, which  is the $n$th degree monic orthogonal polynomial on the unit circle for the weight $\prod_{k=1}^s|1-z\cj{a}_k|^{-r_k}$ obtained by restricting \eqref{bsweights} to $\T$.  

If every $r_k$ is an even, non-positive integer, then $v^*/q$ has for singularities a pole at each  $a_k$, and so $\rho_w=\rho_a$. These poles turn out to be simple if we choose $r_k=0$ for all $k$, in which case the integral in \eqref{polymainformula} reduces to 
\begin{align}\label{anotherintegral}
	\gamma_np_n(z)={} &\frac{1}{2\pi i}\int_{|\zeta|=r}L(z,\zeta)\zeta^{n+1}(1+H_n(\zeta)
	)d\zeta\,.
\end{align}

By a residue calculation, we find from  \eqref{anotherintegral} that if $|z|< 1$ and  $z\not=a_k$ for all $1\leq k\leq s$, then
\begin{align*}
	\gamma_np_n(z)={} &(n+1)z^{n}(1+H_n(z))+z^{n+1}H'_n(z)\\
	&+z^{n+1}(1+H_n(z))\sum_{k=1}^s\frac{\frac{m_k}{2}(1-|a_k|^2) }{(z-a_k)(1-z\cj{a}_k)}\\
	&+	\sum_{k=1}^sa_k^{n+1}(1+H_n(a_k))\left(\frac{\frac{m_k}{2}(1-|a_k|^2) }{(a_k-z)(1-z\cj{a}_k)}+\frac{J(z,a_k)}{q^*(z)q'(a_k)}\right).
\end{align*}

This implies that \eqref{strongasymptoticformula} holds uniformly on any closed subset of $|z|\geq \rho_a$ that contains no $a_k$, while 
\begin{align*}
	p_n(z)={} &	\sum_{k:|a_k|=\rho_a}a_k^{n+1}\left(\frac{\frac{m_k}{2}(1-|a_k|^2) }{(a_k-z)(1-z\cj{a}_k)}+\frac{J(z,a_k)}{q^*(z)q'(a_k)}\right)+O(\rho_a^n/n)
\end{align*}
uniformly as $n\to\infty$ on compact subsets of $|z|<\rho_a$.

For $z=a_j$ with $|a_j|=\rho_a$, a similar residue calculation gives
\[
p_n(a_j)=(1+m_j/2)a_j^{n}(1+O(1/n))\,.
\]

If for some $k$, the number $r_k$ is not an even integer, then $v^*/q$ has a branch point at $a_k$, and the asymptotic analysis can be carried out   with the ideas sketched in Subsection \ref{cases=0}. 

\section{Reproducing kernel}\label{reproducingkernelsection}

For the inner product \eqref{innerproductw} corresponding to a weight $w$ on $\D$, let us set $\|f\|_w:=\sqrt{\langle f,f\rangle_w}$. The Bergman space $A^2_w$ is the linear space of functions  $f\in \caliH(\D)$ such that $\|f\|_w<\infty$, endowed with the inner product \eqref{innerproductw}. We will suppress the index $w$ and simply write $\|\cdot\|$ and $A^2$ for  $w\equiv 1$. 

Suppose that a weight $w$ is such that for some constants $0<\delta<1$ and $m>0$,  it is true that $w(z)\geq m$ whenever $\delta\leq |z|<1$. Then, with
\[
C_{w}:=1+\frac{4\int_{|z|\leq \delta}w(z)d\sigma(z)}{m(1-\delta)^2},
\]
we have
\begin{align}\label{subharmonicineq}
\int_{\D}|f(z)|^2w(z)d\sigma(z)\leq C_{w} \int_{\delta<|z|<1}|f(z)|^2w(z)d\sigma(z)\,,\quad f\in \caliH(\D).
\end{align}

We verify this inequality by closely following the argument presented at the end of Page 2447 of \cite{SSST}. For $f\in \caliH(\D)$, let $z_0$ be a point in the circle $|z|=(\delta+1)/2$ such that $\max_{|z|\leq (\delta+1)/2}|f(z)|=|f(z_0)|$. By the subharmonicity of $|f(z)|^2$ and the fact that the disk  $|z-z_0|<(1-\delta)/2$ is contained in the annulus $\delta<|z|<1$, we have 
\begin{align*}
|f(z_0)|^2\leq{} & \frac{4}{(1-\delta)^2}\int_{|z-z_0|<\frac{1-\delta}{2}}|f(z)|^2d\sigma(z)\\
\leq{} & \frac{4}{m(1-\delta)^2}\int_{\delta<|z|<1}|f(z)|^2w(z)d\sigma(z).
  \end{align*}
From this inequality and the Maximum Modulus Principle, we obtain \eqref{subharmonicineq} at once:
\begin{align*}
		\int_{\D}|f(z)|^2w(z)d\sigma(z)={} & 	\int_{|z|\leq \delta}|f(z)|^2w(z)d\sigma(z)+	\int_{\delta<|z|<1}|f(z)|^2w(z)d\sigma(z)\\
	\leq {} &|f(z_0)|^2\int_{|z|\leq \delta}w(z)d\sigma(z)+	\int_{\delta<|z|<1}|f(z)|^2w(z)d\sigma(z)\\ \leq{} & C_{w}\int_{\delta<|z|<1}|f(z)|^2w(z)d\sigma(z).
\end{align*}

From this point onward, $w$ will represent a weight of the form \eqref{defweightw}, and $(p_n)$ is the corresponding sequence of orthonormal polynomials satisfying \eqref{orthoconditions}. A sequence of functions that converges uniformly on every compact subset of an open set $U$ will be said to be normally convergent in $U$, and this type of convergence will be called normal convergence. 

If $\delta$ is such that $\max_{1\leq k\leq s}|a_k|<\delta< 1$, then there are positive constants $m$ and $M$ such that $m\leq w(z)\leq M$ whenever $\delta\leq |z|<1$. This fact and the inequality \eqref{subharmonicineq} can be utilized to easily deduce that the norms $\|\cdot\|$ and $\|\cdot\|_w$ are equivalent on $\caliH(\D)$, and therefore,  $A^2$ and $A^2_w$ contain the same members. Moreover, it is now easy to  transfer some of the properties of $A^2$ to $A^2_w$. For instance, since $A^2$ is a Hilbert space \cite[p. 8]{DSch}, so is $A^2_w$. We also know that the polynomials are dense in $A^2$ \cite[p. 11]{DSch}, and convergence in $A^2$ implies normal convergence in $\D$, so that the same is true in $A^2_w$, and consequently, we have the (generalized) Fourier expansion identity
\begin{align}\label{fourierconvergence}
	f(z)=\sum_{k=0}^\infty\langle f,p_n\rangle_wp_n(z),\quad f\in A^2_w,
\end{align} 
with the series converging normally in $\D$.

More can actually be said. If $0\leq \varrho(f)\leq 1$ denotes the smallest number such that $f$ admits an analytic continuation to $|z|<1/\varrho(f)$, then   
\begin{align*}
\limsup_{n\to\infty}|\langle f,p_n\rangle_w|^{1/n}=\varrho(f)
\end{align*}
and the normal convergence of \eqref{fourierconvergence} extends to the disk $|z|<1/\varrho(f)$. This stronger statement was proven by Walsh \cite[pp. 130-131]{Walsh} for strictly positive continuous weights over Jordan domains, but the proof still applies (with only minor modifications) to continuous weights that are positive near the boundary of the domain (see also \cite[p. 336]{PS}). 

Because point evaluation functionals are bounded in $A^2$ (see, e.g. \cite[Ch. 1]{DSch}), the same is true in $A^2_w$, and there exists a reproducing kernel $K_w(z,\zeta)$ defined on $\D\times\D$,  analytic in $z$ and $\cj{\zeta}$, such that  for all $f\in A^2_w$ and $z\in \D$,
\begin{align*}
	f(z)=\int_\D f(\zeta)K_w(z,\zeta)w(\zeta)d\sigma(\zeta).
\end{align*}
This kernel has the symmetry property $K(z,\zeta)=\cj{K(\zeta,z)}$, so that the variables can be interchanged at will.

Similarly, the space $A^2_h$ corresponding to the weight 
\[
h(z)=\prod_{k=1}^s\left|\frac{z-a_k}{1-\cj{a}_kz}\right|^{m_k}
\]
has a reproducing kernel $K_h(z,\zeta)$, and by direct integration one can verify that 
\begin{align}\label{kwandkh}
	K_w(z,\zeta)=\frac{K_{h}(z,\zeta)}{v(z)\cj{v(\zeta)}}.
\end{align}

If we now let $\tau_\zeta$ denote, for $\zeta\in \D$, the smallest number such that $K(z,\zeta)$, as a function of $z$, has an analytic continuation to $|z|<1/\tau_\zeta$, then by \eqref{fourierconvergence} we see that for every fixed $\zeta\in \D$,
\begin{align}\label{kernelexpansionpolynomials}
	\frac{K_{h}(z,\zeta)}{v(z)\cj{v(\zeta)}}=\sum_{k=0}^\infty p_k(z)\cj{p_k(\zeta)},\quad |z|<1/\tau_\zeta,
\end{align}
with the convergence of the series being uniform on every disk $|z|\leq r$ of radius $r<1/\tau_\zeta$. Moreover, 
\[
\limsup_{n\to\infty}|p_n(\zeta)|^{1/n}=\tau_\zeta, \quad \zeta\in \D.
\] 

\begin{proposition}For all but finitely many $\zeta\in\D$, 
	\[
	\tau_\zeta= \max\{|\zeta|, \rho_w\}.
	\]
\end{proposition} 

\begin{proof} To prove this proposition we make use of a result from \cite{MD1}, where it has been shown that $K_h(z,\zeta)$ is a rational function in $z$ and $\cj{\zeta}$. As a function of $z$ (and for fixed $\zeta\in \D$) it has  a double pole at $1/\cj{\zeta}$ and a simple pole at $\cj{a}^{-1}_k$, $1\leq k\leq s$. More precisely, from the representation in \cite[Thm. 1.1]{MD1} and \eqref{kwandkh} above, we find that $K_w(z,\zeta)$ can be written as
	\begin{align}\label{mainkernelformula}
	K_{w}(z,\zeta)={}	&\frac{\Pi(z,\zeta)}{(1-z\cj{\zeta})^2q^*(z)v(z)\cj{q^*(\zeta)}\cj{v(\zeta)}},
\end{align}
the numerator having the structure  
\begin{align}\label{formofPi}
	\begin{split}
	\Pi(z,\zeta)={}&\cj{q^*(\zeta)}q^*(z)+(1-z\cj{\zeta})\sum_{k=1}^s\frac{m_k}{2}(1-|a_k|^2)\cj{q^*_k(\zeta)}q_k^*(z)\\
	&+(1-z\cj{\zeta})^2\Pi_1(z,\zeta)\,,
	\end{split}
\end{align}
where
\[
q_k(z):=\underset{j\not=k}{\prod_{1\leq j\leq s}}(z-a_j),\quad q^*_k(z):=\underset{j\not=k}{\prod_{1\leq j\leq s}}(1-\cj{a}_jz)\,,
\]
and $\Pi_1(z,\zeta)$ is  a polynomial in $z$ and $\cj{\zeta}$.  

Suppose first that $\rho_w<|\zeta|<1$. Since $\rho_w$ is the smallest number such that the function $(v(z)q^*(z))^{-1}$ is analytic in $|z|<{1/\rho_w}$, \eqref{mainkernelformula} implies that  $K_w(\cdot,\zeta)$ is analytic in the punctured disk $\{z:|z|<1/\rho_w,\ z\not=1/\cj{\zeta}\}$, with a pole at $z=1/\cj{\zeta}$ (and therefore $\tau_\zeta=|\zeta|$) except possibly when $\zeta=a_k$ for some $1\leq k\leq s$. 

Assume now that $|\zeta|\leq \rho_w$, so that $K_w(\cdot,\zeta)$ is analytic in $|z|<1/\rho_w$, and consequently, $\tau_\zeta\leq \rho_w$. If $\tau_\zeta<\rho_w$, then again by \eqref{mainkernelformula}, $(vq^*)^{-1}$ has to be a meromorphic function in $|z|<\tau_\zeta$, and therefore it has only finitely many singularities (all of them poles) on the circle $|z|=1/\rho_w$. Moreover, at any one of these singularities $z_0$, we must have $\Pi(z_0,\zeta)=0$. Then, to finish the proof it suffices to show that the polynomial $\Pi(z_0,\zeta)$ in the variable $\cj{\zeta}$ is not constant zero, for this will let us conclude that for only finitely many $\zeta$ it is possible to have $\tau_\zeta<\rho_w$. 

From \eqref{formofPi}, we see that if $z_0$ is not a zero of $q^*$, then $$\Pi(z_0,1/\cj{z}_0) =\cj{q^*(1/\cj{z}_0)}q^*(z_0)\not=0.$$ If $q^*(z_0)=0$, say $z_0=1/\cj{a}_k$, then 
\[
\left.\frac{\partial}{\partial \cj{\zeta}}\Pi(z_0,\zeta)\right|_{\zeta=1/\cj{z}_0}=-\frac{z_0m_k}{2}(1-|a_k|^2)\cj{q^*_k(1/\cj{z}_0)}q_k^*(z_0)\not=0.
\]

\end{proof}

 \section{Preliminary results} 

\subsection{Estimates for the Faber polynomials}

By the definition of $\rho_w$, the function $v^*/q$ is analytic in $|z|>\rho_w$, and so is therefore $v^*$.

The Laurent expansion of $v^*$ in $|z|>\rho_w$ is of the form 
\[
v^*(z)=c_0+\frac{c_1}{z}+\frac{c_2}{z^2}+\cdots,\quad c_0=v(0)^{-1},
\]
so that for every integer $n\geq 0$, 
\[
z^nv^*(z)=c_0z^n+c_1z^{n-1}+c_2z^{n-2}+\cdots+c_n+\frac{c_{n+1}}{z}+\frac{c_{n+2}}{z^2}+\cdots.
\]

The polynomial part 
\[
F_n(z):=c_0z^n+c_1z^{n-1}+c_2z^{n-2}+\cdots+c_n
\]
is known as the Faber polynomial for $z^nv^*(z)$ (see, e.g., \cite[Ch. II, \S 4]{Suetin1}).  Note that   
\begin{align}\label{FaberInfinityBehavior}
	F_n(z)=z^nv^*(z)+O(1/z)\quad (z\to\infty).
\end{align}

For every $\rho$ such that $\rho_w<\rho<1$, we also have \cite[Ch. II, \S 4, Eq. 4]{Suetin1} 
\begin{equation*}
F_n(z)=z^nv^*(z)+\frac{1}{2\pi i}\int_{|\zeta|=\rho}\frac{\zeta^nv^*(\zeta)}{\zeta-z}d\zeta, \quad |z|> \rho,
\end{equation*}
and thus if  $\rho_w<\rho_1<\rho<1$, then as $n\to \infty$
\begin{align}\label{FaberBound2}
F_n(z)=z^nv^*(z)+O(\rho_1^n), \quad |z|\geq \rho.
\end{align}

\begin{lemma}\label{estimateofgrowthinannulus}Suppose  that $S(z)$ is a function analytic in some open set containing the annulus $r\leq |z|\leq 1$, and that $|S(z)|=1$ when $|z|=1$.  Then, there exists a constant $M$ such that 
\[
|1-|S(z)|^2|\leq M(1-|z|),\quad r\leq |z|\leq 1.
\]	
\begin{proof} 
Let 
 \[
 M_1:=\max_{r\leq |z|\leq 1}|S(z)|,\quad M_2:=\max_{r\leq |z|\leq 1}|S'(z)|.
\]	

For all $z$ is the annulus $r\leq |z|\leq 1$, the segment $[z,z/|z|]$ is also contained in that annulus, so that
\begin{align*}
	\left|1-|S(z)|^2\right|= {} &(1+|S(z)|)\left||S(z/|z|)|-|S(z)|\right|\\
	\leq {}&(1+M_1)\left|S(z/|z|)-S(z)\right|\\
	\leq {} &(1+M_1)\left|\int_{[z,z/|z|]}S'(\zeta)d\zeta\right|	\leq  (1+M_1)M_2(1-|z|).
\end{align*}
\end{proof}

	\end{lemma}

\begin{proposition}As $n\to\infty$,
\begin{align}\label{normoffaberpolys}
\int_{\D}|F_n(z)|^2w(z)d\sigma = {} &\frac{1}{n}\left(1+O\left(n^{-1}\right)\right)\,.
\end{align}
\end{proposition}
\begin{proof}
For $\rho$ in the  range $\rho_w<\rho<1$, we can write
\begin{align*}
\int_{\D}|F_n(z)|^2w(z)d\sigma(z)={} &\int_{|z|\leq \rho}|F_n(z)|^2w(z)d\sigma(z)+\int_{\rho\leq|z|\leq 1}|F_n(z)|^2w(z)d\sigma(z).
\end{align*}

The first of these two summands can be estimated with the help of \eqref{FaberBound2} and the maximum modulus principle: 
\begin{align}\label{firstestimate}
\int_{|z|\leq \rho}|F_n(z)|^2w(z)d\sigma(z)= O(\rho^{2n})\qquad (n\to\infty).
\end{align}

The function 
\[
S(z):=\frac{v(z)v^*(z)}{\prod_{k=1}^s(1-z\cj{a}_k)^{m_k/2}} \prod_{k=1}^s\left(\frac{z-a_k}{z}\right)^{m_k/2}
\]
has a single-valued analytic branch in the annulus $\rho_w<|z|<1/\rho_w$, and $|S(z)|=1$ for $|z|=1$. Note also that  
\[
|v^*(z)|^2w(z)=|z|^{m}|S(z)|^2, \quad \rho_w<|z|<1/\rho_w\,,
\]
where $m:=\sum_{k=1}^sm_k$.  This equality and \eqref{FaberBound2} imply that  as $n\to\infty$,
\begin{align}\label{secondestimate}
\int_{\rho\leq|z|\leq 1}|F_n(z)|^2w(z)d\sigma(z) =\int_{\rho\leq|z|\leq 1}|z|^{2n+m}|S(z)|^2d\sigma(z)+O(\rho^{2n})\,.
\end{align}

 By Lemma \ref{estimateofgrowthinannulus}, there exists a constant $M$, depending only on $S$ and $\rho$, such that whenever $2n+2>m$,  
\begin{align*}
\int_{\rho\leq|z|\leq 1}&|z|^{2n+m}|S(z)|^2d\sigma(z)\\
&\leq \int_{\rho\leq|z|\leq 1}|z|^{2n+m}\left||S(z)|^2-1\right|d\sigma(z)+\int_{\rho\leq|z|\leq 1}|z|^{2n+m}d\sigma(z)\\
&\leq  M\int_{\rho\leq|z|\leq 1}|z|^{2n+m}(1-|z|^2)d\sigma(z)+\frac{1-\rho^{2n+m+2}}{n+1+m/2}\\
&=  M\left( \frac{1-\rho^{2n+m+2}}{n+1+m/2}- \frac{1-\rho^{2n+m+4}}{n+2+m/2}\right)+ \frac{1-\rho^{2n+m+2}}{n+1+m/2}\\
&=  \frac{1}{n+1+m/2}+O(n^{-2})\,. 
\end{align*}

Thus, we have proven that \eqref{secondestimate} is equal to  $n^{-1}(1+O(n^{-1}))$, and that \eqref{firstestimate} decreases geometrically fast, so that \eqref{normoffaberpolys} takes place. 
\end{proof}

The following observation is not about the Faber polynomials per se, but it will be needed soon in the next section.

Suppose that 
$E_n(z)=e_0+e_1z+\cdots+e_nz^n$ is a polynomial of degree $n$, and let $0\leq r<1$. Since 
\[
\int_{|z|=1}|E_n(z)|^2|dz|=2\pi \sum_{k=0}^n|e_k|^2,  
\] 
and
\[
\int_{r\leq |z|\leq 1}|E_n(z)|^2d\sigma(z)= \sum_{k=0}^n\frac{(1-r^{2k+2})|e_k|^2}{k+1},
\]
it follows that 
\begin{align} \label{norm-inequalities}
\int_{|z|=1}|E_n(z)|^2|dz|\leq \frac{2\pi (n+1)}{1-r^2} \int_{r\leq |z|\leq 1}|E_n(z)|^2d\sigma(z)\,.
\end{align}

\subsection{Behavior of $\gamma_n$ and related quantities} 
\begin{proposition} \label{prop-for-gammas} As $n\to\infty$,
\begin{align}\label{asymptforgamma_n}
 \gamma_n=c_0\sqrt{n}\left(1 +O\left(n^{-1}\right)\right).
\end{align}
\end{proposition}

 \begin{proof}Among all monic polynomials $P$ of degree $\leq n$, the orthogonal polynomial $\gamma_n^{-1}p_n$ is the one with smallest $\|\cdot\|_w$ norm. This follows from the identity $$\|P\|^2_w=\|P-\gamma_n^{-1}p_n\|^2_w+\|\gamma_n^{-1}p_n\|_w^2.$$
 	 	
 Because of this extremality property, we have 
 \begin{align*}
 	1={} & \int_{\D}|p_n(z)|^2w(z)d\sigma(z)\leq \frac{\gamma^2_n}{c_0^2}\int_{\D}|F_n(z)|^2w(z)d\sigma(z),
 \end{align*}
which by \eqref{normoffaberpolys} yields
\begin{align}\label{firstgammainequality}
	1\leq  \frac{\gamma^2_n}{c_0^2n}\left(1 +O(n^{-1 })\right)\,. 
\end{align}

Let $2m_0$ be an even integer such that $2m_0\geq m_k$ for all $1\leq k\leq s$. Since 
\[
\left|\frac{z-a_k}{1-\cj{a}_kz}\right|< 1, \quad |z|<1, 
\]
we have that with $q$ and $q^*$ as defined by \eqref{defqandq*} (recall  \eqref{defweightw}),
\[
w(z)\geq \left|\frac{v(z)q(z)^{m_0}}{q^*(z)^{m_0}}\right|^2,\quad |z|<1.
\] 
Hence 
\begin{align}\label{inequality1}
	1={} & \int_{\D}|p_n(z)|^2w(z)d\sigma(z)\geq \int_{\D}\left|p_n(z)q(z)^{m_0}\cj{g(z)}\right|^2d\sigma(z),
	\end{align}
with
\[
g(z)=\frac{v(z)}{q^*(z)^{m_0}}.
\]

This function $g$ is analytic in $\cj{\D}$, and from its Maclaurin expansion 
\[
g(z)=v(0)+\sum_{k=1}^\infty g_kz^k
\] 
we get   (recall that $c_0=v^*(\infty)=1/v(0)>0$)
\begin{align}\label{expansionofG}
\cj{g(re^{i\theta})}=c_0^{-1}+\sum_{k=1}^\infty \cj{g_k}r^ke^{-ik\theta}.
\end{align}

The polynomial $q(z)^{m_0}$ is a monic polynomial of degree $sm_0$,  so  that if we write 
\[
p_n(z)q(z)^{m_0}=\sum_{k=0}^{n+sm_0}b_{n,k}z^{n+sm_0-k},
\]
then
\[
e^{-i(n+sm_0)\theta}p_n(re^{i\theta})q(re^{i\theta})^{m_0}=\gamma_nr^{n+sm_0}+\sum_{k=1}^{n+sm_0}b_{n,k}r^{n+sm_0-k}e^{-ik\theta}\,.
\]

This and \eqref{expansionofG} yields
\begin{align*}
\int_{0}^{2\pi}\left|e^{-i(n+sm_0)\theta}p_n(re^{i\theta})q(re^{i\theta})^{m_0}\cj{g(re^{i\theta})}\right|^2d\theta \geq 2\pi c_0^{-2}\gamma_n^2r^{2n+2sm_0},
\end{align*}
and so we obtain  from \eqref{inequality1} 
\begin{align*}
1\geq{} &\frac{1}{\pi}\int_{0}^1rdr\int_{0}^{2\pi}\left|e^{-i(n+sm_0)\theta}p_n(re^{i\theta})q(re^{i\theta})^{m_0}\cj{g(re^{i\theta})}\right|^2d\theta\\
\geq {} &\frac{\gamma_n^2c_0^{-2}}{n+sm_0+1}\,.
\end{align*}
This inequality and  that of \eqref{firstgammainequality} combine to produce  \eqref{asymptforgamma_n}.
\end{proof}

Let us now look into the quantities
\begin{align} \label{defalphas}
\alpha_{n,k}:=\frac{1}{2\pi i}\int_{|\zeta|=1}v(\zeta)\zeta^{n-1}\cj{p_k(\zeta)}d\zeta\,, \quad n,k\geq 0.
\end{align}

These numbers $\alpha_{n,k}$ will play an important role in the proof of Theorem \ref{maintheorem}. Since $\zeta=1/\cj{\zeta}$ when $|\zeta|=1$, and since $z^k\cj{p_k(1/\cj{z})}$ is a polynomial whose value at zero is $\gamma_k$, it follows directly from \eqref{defalphas} by an application of Cauchy's integral formula that 
\begin{align} \label{prop-estimates-alphas}
	\alpha_{n,k}= {} &\begin{cases}
		0,&\quad  0\leq k\leq n-1,\\
		\gamma_n/c_0, & \quad k=n.
	\end{cases}
\end{align}

We shall presently see that the numbers $\alpha_{n,k}$ are uniformly bounded once $k>n$, but a better estimate will be achieved later in \eqref{bestestimatealpha}.

\begin{proposition}\label{prop-for-alphas}There exists a constant $C$ such that 
\begin{align}\label{constantC}
|\alpha_{n,k}|\leq C\,,\quad k>n\geq 0\,.
\end{align}
\end{proposition}
\begin{proof}
From the definition \eqref{defalphas}, we see that 
\begin{align*}
\cj{\alpha_{n,k}}= {} &-\frac{1}{2\pi i}\cj{\int_{|\zeta|=1}v(\zeta)\zeta^{n-1}\cj{p_{k}(\zeta)}d\zeta}\\
={} &-\frac{1}{2\pi i}\int_{|\zeta|=1}\cj{v(\zeta)\zeta^{n-1}}p_{k}(\zeta)\cj{d\zeta}\\
={} &\frac{1}{2\pi i}\int_{|\zeta|=1}\cj{v(\zeta)\zeta^{n+1}}p_{k}(\zeta)d\zeta\,.
\end{align*}

Since $\cj{v(z)}=\cj{v(1/\cj{z})}=1/v^*(z)$ for $|z|=1$, this can also be written as 
\begin{align}\label{secondformulaforthealphas}
\cj{\alpha_{n,k}}= {} &\frac{1}{2\pi i}\int_{|\zeta|=1}\frac{\zeta^{-n-1}}{v^*(\zeta)}p_{k}(\zeta)d\zeta\,.
\end{align}

To prove \eqref{constantC}, we first observe that by \eqref{FaberInfinityBehavior}, if $k> n\geq 0$, then
\begin{align*}
\frac{1}{2\pi i}\int_{|\zeta|=1}\cj{v(\zeta)\zeta^{n+1}}F_{k}(\zeta)d\zeta={} &\frac{1}{2\pi i}\int_{|\zeta|=1}\frac{\zeta^{-n-1}}{v^*(\zeta)}\left(\zeta^kv^*(\zeta)+O(1/\zeta)\right)d\zeta\\={} &0.
\end{align*}
 Therefore, for all $ k> n\geq 0$, we have
\begin{align*}
\cj{\alpha_{n,k}}= {} &\frac{1}{2\pi i}\int_{|\zeta|=1}\cj{v(\zeta)\zeta^{n+1}}\left(p_{k}(\zeta)-\frac{\gamma_{k}}{c_0}F_{k}(\zeta)\right)d\zeta,
\end{align*}
and consequently, 
\begin{align}\label{boundformodulusofalpha}
	|\alpha_{n,k}|\leq{}  &\left(\int_{|\zeta|=1}|v(\zeta)|^2|d\zeta|\right)^{1/2}\left(\int_{|\zeta|=1}\left|p_{k}(\zeta)-\frac{\gamma_{k}}{c_0}F_{k}(\zeta)\right|^2|d\zeta|\right)^{1/2}.
\end{align}

We then proceed to estimate the second factor in \eqref{boundformodulusofalpha}. Fix a number $r<1$  such that every $a_k$ is contained in the disk $|z|<r$. Noting that 
\[
d_r:=\min_{r\leq |z|\leq 1}w(z)>0,
\]
and aided by \eqref{norm-inequalities}, we obtain
\begin{align*}\begin{split}
	\int_{|\zeta|=1}\left|p_{k}(\zeta)-\frac{\gamma_{k}}{c_0}F_{k}(\zeta)\right|^2|d\zeta|\leq {} &\frac{2\pi k}{1-r^2}\int_{r\leq |\zeta|\leq 1}\left|p_{k}(\zeta)-\frac{\gamma_{k}}{c_0}F_{k}(\zeta)\right|^2d\sigma(z)\\ 
	\leq {} &\frac{2\pi k}{(1-r^2)d_r}\int_{\D}\left|p_{k}(\zeta)-\frac{\gamma_{k}}{c_0}F_{k}(\zeta)\right|^2w(z)d\sigma(z).
	\end{split}
\end{align*}

The latter integral can be estimated by using \eqref{normoffaberpolys} and \eqref{asymptforgamma_n} as follows:
\begin{align*}\begin{split}
	\int_{\D}&\left|p_{k}(\zeta)-\frac{\gamma_{k}}{c_0}F_{k}(\zeta)\right|^2w(z)d\sigma(z)\\
	={} &1-2\Re\left(\int_{\D}p_{k}(\zeta)\cj{\frac{\gamma_{k}}{c_0}F_{k}(\zeta)}w(z)d\sigma(z)\right)+\frac{\gamma^2_{k}}{c^2_0}\int_{\D}\left|F_{k}(\zeta)\right|^2w(z)d\sigma(z)\\
	={} & \frac{\gamma^2_{k}}{c^2_0}\int_{\D}\left|F_{k}(\zeta)\right|^2w(z)d\sigma(z)-1=O\left(k^{-1}\right)\,.
	\end{split}
\end{align*}

This estimate and the previous inequality show that the right-hand side of \eqref{boundformodulusofalpha} is uniformly bounded in $k$, which proves \eqref{constantC}.
\end{proof}

\section{Proof of Theorem \ref{maintheorem}}\label{sectionproofmaintheorem}
From the structural formula for $K_h$ established in \cite{MD1}, it follows that 
\begin{align}\label{KandL2}
	K_{h}(z,1/\cj{\zeta})=\zeta^2L(z,\zeta),
\end{align}
where $L(z,\zeta)$ is a rational function of the form (recall \eqref{defqandq*})
\begin{align}\label{structureofL}
L(z,\zeta)=\frac{1}{(\zeta-z)^2} +\frac{V(z,\zeta) }{(\zeta-z)q^*(z)q(\zeta)}
\end{align}
with $V(z,\zeta)$ a polynomial in the variables $z$ and $\zeta$.

By the  symmetry of the kernel, the series in \eqref{kernelexpansionpolynomials} converges locally uniformly for $|\zeta|<1/\tau_z$, for each fixed $z\in \D$. Since $\tau_z<1$, we can combine \eqref{kernelexpansionpolynomials}, \eqref{KandL2} and \eqref{exteriorszegofunction} to obtain 
\begin{align}\label{expansionwithL}
\frac{(vv^*)(\zeta)L(z,\zeta)\zeta^{n+1}}{v(z)}=v(\zeta)\zeta^{n-1}\sum_{k=0}^\infty \cj{p_k(1/\cj{\zeta})}p_k(z)\,, \quad z\in \D,\ |\zeta|=1.
\end{align}

Let us now define, for every $n\geq 0$, a function $Q_n$ analytic in $|z|<1/\rho_w$ as follows.  For all $\rho_w<r<1/\rho_w$,
\begin{align}\label{defQ}
Q_n(z):=\frac{1}{2\pi i v(z)}\int_{|\zeta|=r}(vv^*)(\zeta)L(z,\zeta)\zeta^{n+1}d\zeta,\quad |z|<r.
\end{align}

By Cauchy's theorem, it is clear that this function is well-defined, meaning that for $z$ fixed, the integral in \eqref{defQ} produces the same value for any $r>|z|$.

If we make use of \eqref{defalphas} and \eqref{prop-estimates-alphas} when integrating \eqref{expansionwithL} over the unit circle with respect to $\zeta$, we find that for all $n\geq 0$,
\begin{align}\label{polyfirstexpansion}
	\alpha_{n,n}p_n(z)={} &Q_n(z)-\sum_{k=n+1}^\infty \alpha_{n,k}p_k(z),\quad |z|<1.
\end{align}
 
Our next step is crucial, and consists of iterating \eqref{polyfirstexpansion}. This is  best carried out in two separate lemmas.

For any triplet of integers $k$, $m$, and $n$ such that $m,n\geq 0$ and $k\geq n+m+1$, we  define corresponding numbers $h(n,m)$ and $g(n,m,k)$ in the following recursive manner:  
\begin{align}\label{numbers-h-and-g-1}
h(n,0)=1,\quad g(n,0,k)=-\alpha_{n,k},\quad k> n\geq 0,
\end{align}
\begin{align}\label{numbers-h-and-g-2}
h(n,m+1)=\frac{g(n,m,n+m+1)}{\alpha_{n+m+1,n+m+1}}\,,
\end{align}
\begin{align}\label{numbers-h-and-g-3}
g(n,m+1,k)= g(n,m,k)-h(n,m+1)\alpha_{n+m+1,k},\quad k\geq n+m+2.
\end{align}

\begin{lemma}
For all integers $n,m\geq 0$, the equality
\begin{align}\label{polysecondexpansion}
\alpha_{n,n}p_n(z)={} &\sum_{j=0}^mh(n,j)Q_{n+j}(z)+\sum_{k=n+m+1}^\infty g(n,m,k)p_k(z)
\end{align}
holds true for $|z|<1$.
\end{lemma}
\begin{proof} The proof is by induction on $m$. For $m=0$, \eqref{polysecondexpansion} reduces to \eqref{polyfirstexpansion}. Assuming \eqref{polysecondexpansion} is true for some $m$, we can use \eqref{polyfirstexpansion} (with $n$ replaced by $n+m+1$) to get 
\begin{align*}
\alpha_{n,n}p_n(z)={} &\sum_{j=0}^mh(n,j)Q_{n+j}(z)+\sum_{k=n+m+2}^\infty g(n,m,k)p_k(z)\\
&+\frac{ g(n,m,n+m+1)}{\alpha_{n+m+1,n+m+1}}\alpha_{n+m+1,n+m+1}p_{n+m+1}(z)\\
={} &\sum_{j=0}^mh(n,j)Q_{n+j}(z)+\sum_{k=n+m+2}^\infty g(n,m,k)p_k(z)\\
&+\frac{ g(n,m,n+m+1)}{\alpha_{n+m+1,n+m+1}}\left(Q_{n+m+1}(z)-\sum_{k=n+m+2}^\infty\alpha_{n+m+1,k}p_k(z)\right)\\
={}&\sum_{j=0}^mh(n,j)Q_{n+j}(z)+\frac{ g(n,m,n+m+1)}{\alpha_{n+m+1,n+m+1}}Q_{n+m+1}(z)\\
&+\sum_{k=n+m+2}^\infty\left( g(n,m,k)-\frac{ g(n,m,n+m+1)}{\alpha_{n+m+1,n+m+1}}\alpha_{n+m+1,k}\right)p_k(z)\,.  
\end{align*}

The latter three terms are precisely what we obtain if we replace $m$ by $m+1$ in \eqref{polysecondexpansion} and use the defining relations \eqref{numbers-h-and-g-2} and \eqref{numbers-h-and-g-3}.
\end{proof}

\begin{lemma}\label{lemma-poly-expansion}Suppose that there exist constants $C,C',\lambda\geq 0$ such that 
\begin{align}\label{bound-conditions}
\frac{1}{\alpha_{n,n}}\leq \frac{C'}{\sqrt{n}},\quad |\alpha_{n,k}|\leq \frac{C}{k^\lambda},\quad k>n\geq 0.
\end{align}
Then, we have the series representation 
\begin{align}\label{polythirdexpansion}
\alpha_{n,n}p_n(z)={} &\sum_{j=0}^\infty h(n,j) Q_{n+j}(z), \quad |z|<1,
\end{align}
with $h(n,0)=1$ and the remaining coefficients satisfying the inequality
\begin{align}\label{bounds-for-hn}
|h(n,j)|\leq \frac{CC'}{(n+j)^{\lambda+1/2}}\prod_{\ell=1}^{j-1}\left(1+\frac{CC'}{(n+\ell)^{\lambda+1/2}}\right) ,\quad j\geq 1\,.
\end{align}
\end{lemma}
\begin{proof} By hypothesis, and by the defining relations \eqref{numbers-h-and-g-1}-\eqref{numbers-h-and-g-3}, we have
\[
h(n,0)=1,\quad |g(n,0,k)|\leq  \frac{C}{k^\lambda },\quad k\geq n+1,
\] 
\begin{align*}
|h(n,m+1)|\leq {} &\frac{C'|g(n,m,n+m+1)|}{\sqrt{n+m+1}}\,,\\
|g(n,m+1,k)|\leq {} & |g(n,m,k)|+\frac{C}{k^{\lambda}}|h(n,m+1)|,\quad k\geq n+m+2.
\end{align*}

From these inequalities, we proceed by induction (in the second variable $m$) to prove \eqref{bounds-for-hn} and that  
\begin{align}\label{inequalityforg}
|g(n,m,k)|\leq \frac{C}{k^\lambda}\prod_{\ell=1}^{m}\left(1+\frac{CC'}{(n+\ell)^{\lambda+1/2}}\right),\quad m\geq 0,\ k\geq n+m+1.
\end{align}

For $m=0$ the inequality \eqref{inequalityforg} is correct. Then, assuming  it is also correct for $g(n,m,k)$, we have 
\begin{align*}
	|h(n,m+1)|\leq \frac{C'}{\sqrt{n+m+1}}\frac{C}{(n+m+1)^\lambda}\prod_{\ell=1}^{m}\left(1+\frac{CC'}{(n+\ell)^{\lambda+1/2}}\right)
\end{align*}
and 
\begin{align*}
	|g(n,m+1,k)|\leq{} & \frac{C}{k^\lambda}\prod_{\ell=1}^{m}\left(1+\frac{CC'}{(n+\ell)^{\lambda+1/2}}\right)\\
	&+\frac{C}{k^{\lambda}}\frac{C'C}{(n+m+1)^{\lambda+1/2}}\prod_{\ell=1}^{m}\left(1+\frac{CC'}{(n+\ell)^{\lambda+1/2}}\right)\\
	\leq {} &\frac{C}{k^\lambda}\left(1+\frac{C'C}{(n+m+1)^{\lambda+1/2}}\right)\prod_{\ell=1}^{m}\left(1+\frac{CC'}{(n+\ell)^{\lambda+1/2}}\right)\,.
\end{align*}

Let us now think of $z\in \D$ as being fixed. Since  $\tau_z=\limsup_{k\to\infty}|p_k(z)|^{1/k}<1$, we can find $\epsilon>0$ and a constant $A_\epsilon$ such that $\tau_z+\epsilon<1$ and $|p_k(z)|\leq A_{\epsilon} (\tau_z+\epsilon)^k$ for all $k\geq 0$. This and \eqref{inequalityforg} imply that 
\begin{align}\label{product-convergence}
	\begin{split}
&\sum_{k=n+m+1}^\infty |g(n,m,k)p_k(z)|\\
&\leq   A_\epsilon C\prod_{\ell=1}^{m}\left(1+\frac{CC}{(n+\ell)^{\lambda+1/2}}\right)\sum_{k=n+m+1}^\infty \frac{(\tau_z+\epsilon)^k}{k^\lambda}\\
&{}\leq {} \frac{A_\epsilon C (\tau_z+\epsilon)^{n+1}}{(n+m+1)^\lambda(1-(\tau_z+\epsilon))}\prod_{\ell=1}^{m}\left[\left(1+\frac{CC'}{(n+\ell)^{\lambda+1/2}}\right)(\tau_z+\epsilon)\right].
\end{split}
\end{align}

Since 
\[
\lim_{\ell\to\infty}\left(1+\frac{CC'}{(n+\ell)^{\lambda+1/2}}\right)(\tau_z+\epsilon)= (\tau_z+\epsilon)<1,
\]
the last product that occurs in \eqref{product-convergence} converges to zero as $m\to\infty$, and so we conclude that for $|z|<1$,
\[
\lim_{m\to\infty}\sum_{k=n+m+1}^\infty |g(n,m,k)p_k(z)|=0,
\]
which in view of \eqref{polysecondexpansion} confirms the validity or \eqref{polythirdexpansion}. 
\end{proof}

Let us now define
\[
H_n(z):=\sum_{j=1}^\infty h(n,j) z^j\,.
\]

By Lemma \ref{lemma-poly-expansion}, when constants $C,C',\lambda\geq 0$  can be found such that \eqref{bound-conditions} holds, then (see \eqref{bounds-for-hn})
\[
|H_n(z)|\leq  \frac{CC'}{(n+1)^{\lambda+1/2}}\sum_{j=1}^\infty|z|^{j}\prod_{\ell=1}^{j-1}\left(1+\frac{CC'}{(n+\ell)^{\lambda+1/2}}\right)\,.
\]
This would imply that $H_n(z)$ is well-defined and analytic in $\D$, and that 
\begin{align}\label{estimateforH}
|H_n(z)|=O(n^{-\lambda-1/2})
\end{align}
uniformly as $n\to \infty$ on every disk $|z|\leq r$ of radius $r<1$. Moreover, from the definition of $Q_n$ and \eqref{polythirdexpansion},  we would also  have that for every $r$ in the range $\rho_w<r<1$, 
\begin{align}\label{formula-to-be-improved}
\alpha_{n,n}p_n(z)={} &
\frac{1}{2\pi i v(z)}\int_{|\zeta|=r}(vv^*)(\zeta)L(z,\zeta)\zeta^{n+1}(1+H_n(\zeta))d\zeta,\quad |z|<r.
\end{align}

By \eqref{prop-estimates-alphas} and Proposition \ref{prop-for-gammas}, the first inequality of \eqref{bound-conditions} is always satisfied, and by Proposition \ref{prop-for-alphas}, the second inequality of \eqref{bound-conditions} holds true with $\lambda=0$. Therefore, we can say for sure that $H_n(z)=O(n^{-1/2})$ uniformly as $n\to\infty$ on compact subsets of $\D$.

We now proceed to get a better estimate for the coefficients $\alpha_{n,k}$ of \eqref{defalphas}. Let $\rho_v$ be the smallest number such that $v^*$ can be analytically continued without ever vanishing in $|z|>\rho_v$, and suppose that $\mu$ and $\eta$ are such that 
\begin{align}\label{numbers}
\rho_w<\mu<1, \quad \max\{\mu,\rho_v\}<\eta<1.
\end{align}

Bearing in mind the representation \eqref{structureofL}, the contour of integration in  \eqref{formula-to-be-improved} can be retracted to the circle $|\zeta|=\mu$ if we properly account for residues. Thus, for $\mu<|z|<1$,  we have
\begin{align}\label{closetofinalrepresentation}
	\begin{split}
\alpha_{n,n} p_n(z)={} &R_n(z) +\frac{1}{v(z)}\left[v(z)v^*(z)z^{n+1}(1+H_n(z))\right]    
'\\
&+\frac{z^{n+1} v^*(z)V(z,z) (1+H_n(z)
	)}{q^*(z)q(z)},
\end{split}
\end{align}
with
\begin{align*}
R_n(z)={} &\frac{1}{2\pi i v(z)}\int_{|\zeta|=\mu}(vv^*)(\zeta)L(z,\zeta)\zeta^{n+1}(1+H_n(\zeta))d\zeta\\
={} & \frac{1}{2\pi i v(z)}\int_{|\zeta|=\mu}\frac{v(\zeta)v^*(\zeta)\zeta^{n+1}}{(\zeta-z)^2}(1+H_n(\zeta))d\zeta\\
&+\frac{1}{2\pi i v(z)q^*(z)}\int_{|\zeta|=\mu}\frac{(v^*/q)(\zeta)v(\zeta)V(z,\zeta)\zeta^{n+1}}{\zeta-z}(1+H_n(\zeta))d\zeta\,.
\end{align*}

By Cauchy's theorem, we can write  \eqref{secondformulaforthealphas} as 
\begin{align*}
	\cj{\alpha_{n,k}}= {} &\frac{1}{2\pi i}\int_{|\zeta|=\eta}\frac{\zeta^{-n-1}}{v^*(\zeta)}p_{k}(\zeta)d\zeta\,.
\end{align*}

In this integral, we replace $p_k$ by its representation \eqref{closetofinalrepresentation} to  get  
 \begin{align}\label{betterestimateforalpha-1}
\alpha_{k,k}\cj{\alpha_{n,k}}= {} & I^1_{n,k}+I^2_{n,k}+I^3_{n,k},
\end{align}
where 
 \begin{align}
		I^1_{n,k}= {} & \frac{1}{2\pi i}\int_{|\zeta|=\eta}\frac{\zeta^{-n-1}}{v^*(\zeta)}R_{k}(\zeta)d\zeta\,,\label{defI1}\\
	I^2_{n,k}={}&\frac{1}{2\pi i}\int_{|\zeta|=\eta}\frac{\zeta^{-n-1}\left[v(\zeta)v^*(\zeta)\zeta^{k+1}(1+H_k(\zeta))\right]'}{v(\zeta)v^*(\zeta)}d\zeta\,,\nonumber\\
		I^3_{n,k}={}	& \frac{1}{2\pi i}\int_{|\zeta|=\eta}\frac{\zeta^{k-n} V(\zeta,\zeta) (1+H_k(\zeta)
		)}{q^*(\zeta)q(\zeta)}d\zeta\,.\label{defI3}
	\end{align}

We proceed to estimate $I^1_{n,k}$. Note that $R_k(z)$ is well-defined and analytic for values of $z$ in the annulus  $\mu<|z|<1/\rho_w$, and that $R_k(z)=O(\mu^k)$ uniformly as $k\to\infty$ on any circle $|z|=r$ with $\mu<r<1/\rho_w$. 
Then, the value of the integral in \eqref{defI1} does not change if the circle of integration $|\zeta|=\eta$ is moved to the circle $|\zeta|=1/\mu$, which leads to   
  \begin{align}\label{estimateI1}
 \begin{split}
I^1_{n,k}=O\left(\mu^{k+n}\right),\quad \rho_w<\mu<1,
\end{split}
\end{align}
with the implied constant depending only on the value of $\mu$.

Next, we estimate $I^2_{n,k}$. By a direct computation of the derivative, we find
\begin{align*}
\frac{\left[v(\zeta)v^*(\zeta)\zeta^{k+1}(1+H_k(\zeta))\right]'}{v(\zeta)v^*(\zeta)}={} &(k+1)\zeta^{k}(1+H_k(\zeta)
)+\zeta^{k+1}H_k'(\zeta)\\
&+\frac{v'(\zeta)}{v(\zeta)}\zeta^{k+1}(1+H_k(\zeta))\\
&+\frac{{v^*}'(\zeta)}{v^*(\zeta)}\zeta^{k+1}(1+H_k(\zeta))\,.
\end{align*}

Since $H_k$ and $v$ are analytic in $\D$, we see that for $k> n$, $I^2_{n,k}$ reduces to 
\begin{align}\label{integralI2}
I^2_{n,k}= {} &\frac{1}{2\pi i}\int_{|\zeta|=\eta}\frac{{v^*}'(\zeta)}{v^*(\zeta)}\zeta^{k-n}(1+H_k(\zeta)
)d\zeta\,.
\end{align}

The number $\rho_v$ is also the smallest number such that the function ${v^*}'/v^*$ has an analytic continuation to $|z|>\rho_v$, and so we get from \eqref{integralI2}
\begin{align}\label{estimateI2}
	I^2_{n,k}= {} &O(\eta^{k-n}),\quad \rho_v<\eta<1,
\end{align}
with the implied constant depending only on $\eta$.

By comparing  \eqref{structureofL} and \eqref{secondformulaforL}, we see that  
\begin{align*}
	\frac{V(\zeta,\zeta) }{q^*(\zeta)q(\zeta)}={}&\sum_{j=1}^s\frac{\frac{m_j}{2}(1-|a_j|^2)}{(\zeta-a_j)(1-\zeta\cj{a}_j)}.
\end{align*}

We substitute this expression in \eqref{defI3} and compute the resulting integral by means of the residue theorem to obtain
\begin{align}\label{estimateI3}
	I^3_{n,k}={}	&  \sum_{j=1}^s\frac{m_k}{2}a_j^{k-n}(1+H_k(a_j))=O(\rho_a^{k-n}),
\end{align}
with $\rho_a=\max_{1\leq k\leq s}|a_k|$.

Having established \eqref{estimateI1}, \eqref{estimateI2}, and \eqref{estimateI3}, it follows from \eqref{betterestimateforalpha-1} that with $\mu$ and $\eta$ as specified in \eqref{numbers},
\begin{align*}
	\alpha_{k,k}\cj{\alpha_{n,k}}= {} &O\left(\mu^{k+n}\right)+O(\eta^{k-n})+O(\rho_a^{k-n}).
\end{align*}

Since $v^*/q$ is analytic in $|z|>\rho_w$, so is $v^*$ alone, and thus $\rho_w\leq \rho_v$. With this observation in mind, and recalling \eqref{prop-estimates-alphas} and \eqref{asymptforgamma_n}, we conclude from the latter equality that 
for every $\eta$ in the range $\rho_v<\eta<1$, there exists a constant $C_\eta$ such that 
\begin{align}\label{bestestimatealpha}
|\alpha_{n,k}|\leq \frac{C_\eta (\eta^{k-n}+\rho_a^{k-n})}{\sqrt{k}},\quad k>n\geq 0.
\end{align} 
 
This shows that the second inequality of \eqref{bound-conditions} is satisfied with $\lambda=1/2$. It follows from the discussion leading to \eqref{estimateforH} that Formula \eqref{formula-to-be-improved} holds with  $H_n(z)=O(n^{-1})$, completing the proof of Theorem  \ref{maintheorem}.

From \eqref{closetofinalrepresentation}, we see that for every $\mu$ and $r$ such that $\rho_w<\mu< r<1$, the equality
\begin{align*}
		\frac{\alpha_{n,n} p_n(z)}{(n+1)z^n}={}&O(n^{-1}(\mu/r)^n)+v^*(z)(1+H_n(z))+\frac{zv^*(z)}{n+1}H'_n(z)\\
		&+\frac{z}{n+1}\left(\frac{(vv^*)'(z)}{v(z)}+\frac{ v^*(z)V(z,z) }{q(z)q^*(z)}\right)(1+H_n(z))
\end{align*}
holds uniformly for $|z|=r$ as $n\to\infty$. This quickly leads to \eqref{strongasymptoticformula} for $|z|=r$, and by the Maximum Modulus Principle,  for $|z|\geq r$ as well.

\end{document}